\documentclass[12pt,a4paper,intlimits,sumlimits]{amsart}

\usepackage{amsmath, amsthm, mathtools}
\usepackage{amsfonts}
\usepackage[alphabetic]{amsrefs}
\usepackage{graphicx}
\usepackage{framed}
\usepackage{amssymb}
\usepackage{esvect}
\usepackage{xcolor}
\usepackage{eucal}
\usepackage{float}
\usepackage[normalem]{ulem}

\usepackage{hyperref}
\usepackage[english]{babel}

\def \N {\mathbb{N}}
\def \R {\mathbb{R}}

\newcommand{\Fucik}{Fu\v c\'\i k }
\newcommand{\half}{\frac{1}{2}}

\newcommand{\PS}[1]{$(\text{PS})_{#1}$}
\newcommand{\pnorm}[2][]{\if #1'' \left|#2\right|_p \else \left|#2\right|_{#1} \fi}

\newcommand{\set}[1]{\left\{#1\right\}}

\newcommand\redsout{\bgroup\markoverwith{\textcolor{red}{\rule[0.5ex]{2pt}{0.4pt}}}\ULon}

\theoremstyle{definition}
\newtheorem{definition}{Definition}[section]

\newtheorem{remark}[definition]{Remark}

\theoremstyle{plain}
\newtheorem{theorem}[definition]{Theorem}
\newtheorem{proposition}[definition]{Proposition}
\newtheorem{lemma}[definition]{Lemma}
\newtheorem{corollary}[definition]{Corollary}

\numberwithin{equation}{section}

\usepackage{geometry}
\renewcommand{\le}{\leqslant}

\renewcommand{\ge}{\geqslant}

\allowdisplaybreaks

 \title[Superposition operators of mixed order and jumping nonlinearities]{An existence theory \\
for superposition operators of mixed order\\
subject to jumping nonlinearities}
 
\author{Serena Dipierro}
\address{Serena Dipierro: Department of Mathematics and Statistics, The University of Western Australia, 35 Stirling Highway, Crawley, Perth, WA 6009, Australia}
\email{serena.dipierro@uwa.edu.au}

\author{Kanishka Perera}
\address{Kanishka Perera: Department of Mathematical Sciences, Florida Institute of Technology, 150 W University Blvd, Melbourne, FL 32901-6975, USA}
\email{kperera@fit.edu}

\author{Caterina Sportelli}
\address{Caterina Sportelli: Department of Mathematics and Statistics, The University of Western Australia, 35 Stirling Highway, Crawley, Perth, WA 6009, Australia}
\email{caterina.sportelli@uwa.edu.au}

\author{Enrico Valdinoci}
\address{Enrico Valdinoci: Department of Mathematics and Statistics, The University of Western Australia, 35 Stirling Highway, Crawley, Perth, WA 6009, Australia}
\email{enrico.valdinoci@uwa.edu.au}

 \date{}

\makeindex

\begin{document}
 \maketitle

\begin{abstract}
We consider a superposition operator of the form
$$ \int_{[0, 1]} (-\Delta)^s u\, d\mu(s),$$
for a signed measure~$\mu$ on the interval of fractional exponents~$[0,1]$,
joined to a nonlinearity whose term of homogeneity equal to one is ``jumping'', i.e.
it may present different coefficients in front of the negative and positive parts.

The signed measure is supposed to possess a positive contribution coming from the higher exponents that overcomes its negative contribution (if any).

The problem taken into account is also of ``critical'' type, though in this case the critical exponent
needs to be carefully selected in terms of the signed measure~$\mu$.

Not only the operator and the nonlinearity considered here are very general,
but our results are new even in special cases of interest and include known results as particular subcases.

The possibility of considering operators ``with the wrong sign'' is also a complete novelty in this setting.
\end{abstract}
 
\section{Introduction} 

The aim of this paper is to address the study of critical problems involving a nonlocal operator
obtained through the linear superposition of fractional operators of different orders.

Specifically,
we consider two nonnegative finite (Borel) measures~$\mu^+$ and~$\mu^-$ in~$[{{ 0 }}, 1]$,
as well as the corresponding signed measure~$\mu:=\mu^+ -\mu^-$.

The main operator of interest for us takes the form
\begin{equation}\label{AMMU} A_\mu u:=\int_{[{{ 0 }}, 1]} (-\Delta)^s u\, d\mu(s) .\end{equation}
As customary, the notation~$(-\Delta)^s$ is reserved to the fractional Laplacian, defined, for all~$s\in(0,1)$ as
\begin{equation}\label{AMMU0} (- \Delta)^s\, u(x) = c_{N,s}\int_{\R^N} \frac{2u(x) - u(x+y)-u(x-y)}{|y|^{N+2s}}\, dy.
\end{equation}
The positive normalizing constant~$c_{N,s}$ is chosen in such a way that, for~$u$ smooth and rapidly decaying,
the Fourier transform of~$(-\Delta)^s u$ returns~$(2\pi|\xi|)^{2s}$ times the Fourier transform of~$u$
and provides
consistent limits as~$s\nearrow1$ and as~$s\searrow0$, namely
$$ \lim_{s\nearrow1}(-\Delta)^su=(-\Delta)^1u=-\Delta u
\qquad{\mbox{and}}\qquad
\lim_{s\searrow0}(-\Delta)^s u=(-\Delta)^0u=u.
$$

Particular cases for the operator in~\eqref{AMMU} are (minus) the Laplacian (corresponding to the choice of~$\mu$ being the
Dirac measure concentrated at~$1$), the fractional Laplacian~$(-\Delta)^{s_\star}$
(corresponding to the choice of~$\mu$ being the
Dirac measure concentrated at some fractional power~$s_\star$),
the ``mixed order operator'' $-\Delta+(-\Delta)^{s_\star}$
(when~$\mu$ is the sum of two Dirac measures), etc.

The ``continuous'' superposition of operators of different fractional orders has also been recently considered in the literature, see e.g.~\cite{MR3485125}.

A list of interesting cases for this operator will be discussed in detail in Section~\ref{EXA:A}.

For the moment, let us recall that operators arising from the superpositions of local and nonlocal operators
are a topic intensively studied in the contemporary research, under different perspectives, including
regularity theory (see~\cite{MR2911421, MR4381148, MR4387204, MR4469224, MR4530314, MIN}),
existence and nonexistence results (see~\cite{MR4275496}),
viscosity solution theory (see~\cite{MR2129093, MR2653895}),
symmetry results (see~\cite{MR4313576}), geometric and variational inequalities (see~\cite{MR4391102, FABE}), etc.

Moreover, operators of this type naturally arise in concrete applications: for instance they model the dispersal of a biological population whose individuals are subject to different kinds of diffusive strategies (such as Gaussian and L\'evy flights), see~\cite{MR4249816, EDOARDO}.

Interestingly, the signed measure~$\mu$ allows each single fractional Laplacian to take part in~\eqref{AMMU} with possibly different signs. However, in our setting, we will assume that
the interval~$[{{ 0 }}, 1]$ can be divided into two subintervals, the one on the right of a specific sign
and suitably dominating the one on the left. More precisely,
we assume that there exist~$\overline s\in ({{ 0 }}, 1]$ and~$\gamma\ge0$ such that
\begin{equation}\label{mu00}
{\mu^+}([\overline s, 1])>0,
\end{equation}
\begin{equation}\label{mu3}
{\mu^-}_{\big|_{[\overline s, 1]}}=0
\end{equation}
and
\begin{equation}\label{mu2}
\mu^-\big([{{ 0 }}, \overline s]\big)\le\gamma
\mu^+\big([\overline s, 1]\big).
\end{equation}
Roughly speaking, conditions~\eqref{mu00} and~\eqref{mu3} state that the component of the signed measure~$\mu$ supported
on higher fractional exponents is positive, and condition~\eqref{mu2} prescribes that the
negative components of the signed measure~$\mu$ (if any) must be conveniently ``reabsorbed'' into the positive ones.
Our main assumption will thus be that~$\gamma$ in~\eqref{mu2} is sufficiently small.
\medskip

It is worth pointing out that, by assumption~\eqref{mu00}, 
there exists~$s_\sharp\in [\overline s, 1]$ such that
\begin{equation}\label{scritico}
\mu^+ ([s_\sharp, 1])> 0.
\end{equation}
We will see below that the exponent~$s_\sharp$ plays the role of a critical exponent
(therefore, roughly speaking, while some arbitrariness is allowed in the choice of~$s_\sharp$ here above,
the results obtained will be stronger if one picks~$s_\sharp$ ``as large as possible''
but still fulfilling~\eqref{scritico}).
\medskip

In this paper,
we investigate the existence of nontrivial solutions of the critical growth elliptic problem with a jumping nonlinearity
\begin{equation} \label{mainab}
\left\{\begin{aligned}
\int_{[{{ 0 }}, 1]} (-\Delta)^s u\, d\mu(s) & = bu^+ - au^- + |u|^{2_{s_\sharp}^\ast - 2}\, u && \text{in } \Omega,\\
u & = 0 && \text{in } \R^N \setminus \Omega,
\end{aligned}\right.
\end{equation}
where~$\Omega$ is an open bounded subset of~$\R^N$, with~$N\ge 3$.
Moreover, $$2_{s_\sharp}^\ast = \frac{2N}{N - 2s_\sharp}$$ is the fractional critical Sobolev exponent
corresponding to the parameter~$s_\sharp$, $a$ and~$ b $ are positive real numbers
and~$u^\pm := \max \set{\pm u,0}$ are the positive and negative parts of~$u$, respectively.

When~$a = b $, the right-hand side of the equation in~\eqref{mainab} reduces to a Br{\'e}zis-Nirenberg 
nonlinearity.
When~$a\ne b$, an additional difficulty arises in the analysis of~\eqref{mainab},
due to the lack of regularity of the source term.
As a matter of fact, the study of non-differentiable nonlinearities (sometimes called
``jump'' nonlinearities) has a very consolidated tradition,
see e.g.~\cite{MR499709, MR1181350, MR1215262, MR1725568}.
Interestingly, this type of nonlinearity frequently occurs in concrete problems, including singular perturbations of classical nonlinearities, plasma problems,
mathematical biology, etc., see~\cite{MR1303035}.
\medskip

Our goal in this paper is to find pairs~$(a,b)$ which guarantee the existence of nontrivial solutions for~\eqref{mainab}.
To this end, we need to identify suitable regions of the space of parameters (corresponding to~$\R^2$)
that are conveniently related to the spectral properties of the operator~$A_\mu$ in~\eqref{AMMU}.
To this end, we briefly recall the construction of the minimal and maximal curves of the Dancer-\Fucik spectrum (see~\cite[Chapter 4]{MR3012848} for a general setting). 

One looks at the  operator~$A_\mu$ in~\eqref{AMMU}
(the setting can be actually generalized to include the case of monotone, self-adjoint operators with compact inverse, coupled to potentials).
The classical spectrum of~$A_\mu$ consists of isolated eigenvalues~$\lambda_l$, with~$ l \ge 1$, with finite multiplicity, satisfying~$0 < \lambda_1 < \cdots < \lambda_l < \cdots$. 

Instead, the Dancer-\Fucik spectrum of~$A_\mu$
consists of the couples~$(a,b) \in \R^2$ for which the equation
\begin{equation} \label{2.3}
A_\mu u = bu^+ - au^-
\end{equation}
has a nontrivial solution.

The Dancer-\Fucik spectrum is a closed subset of~$\R^2$ (see~\cite[Proposition 4.4.3]{MR3012848}).
We also point out that equation~\eqref{2.3} reduces to~$A_\mu u = \lambda u$ when~$a = b = \lambda$, and therefore the Dancer-\Fucik spectrum of~$A_\mu$ contains
points of the form~$(\lambda_l,\lambda_l)$.

The Dancer-\Fucik spectrum presents an interesting geometry, see~\cite[Theorem 4.7.9]{MR3012848}.
Namely, there exist two continuous and strictly decreasing functions~$\nu_{l-1}$
and~$\mu_l$, such that:
\begin{itemize}
\item for all~$a\in(\lambda_{l-1},\lambda_{l+1})$, we have that~$\nu_{l-1}(a)\le\mu_l(a)$,
\item $\nu_{l-1}(\lambda_l) = \lambda_l=\mu_l(\lambda_l)$,
\item for all~$a\in(\lambda_{l-1},\lambda_{l+1})$, we have that both~$(a,\nu_{l-1}(a))$ and~$(a,\mu_{l}(a))$
belong to the Dancer-\Fucik spectrum,
\item if~$a\in(\lambda_{l-1},\lambda_{l+1})$ and~$b\in (\lambda_{l-1},\lambda_{l+1})$,
with either~$b < \nu_{l-1}(a)$ or~$b > \mu_l(a)$, then~$(a,b)$ does not belong to the Dancer-\Fucik spectrum.
\end{itemize}

In particular, setting, for any~$ l \ge 2$,
\begin{equation}\label{QELL}
Q_l := (\lambda_{l-1},\lambda_{l+1}) \times (\lambda_{l-1},\lambda_{l+1}),
\end{equation}
we have that the graphs of~$\nu_{l-1}$ and~$\mu_l$
are strictly decreasing curves in~$Q_l$ that belong to the Dancer-\Fucik spectrum.
Also, both these curves pass through the point~$(\lambda_l,\lambda_l)$, while the region~$\set{(a,b) \in Q_l : b < \nu_{l-1}(a)}$ below the lower curve and the region~$\set{(a,b) \in Q_l : b > \mu_l(a)}$ above the upper curve lie outside the Dancer-\Fucik spectrum.

Points in the region~$\set{(a,b) \in Q_l : \nu_{l-1}(a) < b < \mu_l(a)}$ between these two graphs
(when such region is nonempty) may or may not belong to the Dancer-\Fucik spectrum.\medskip

The geometry related to the Dancer-\Fucik spectrum is sketched in Figure~\ref{FI11}.

\begin{figure}[h]
\begin{center}
\includegraphics[scale=.25]{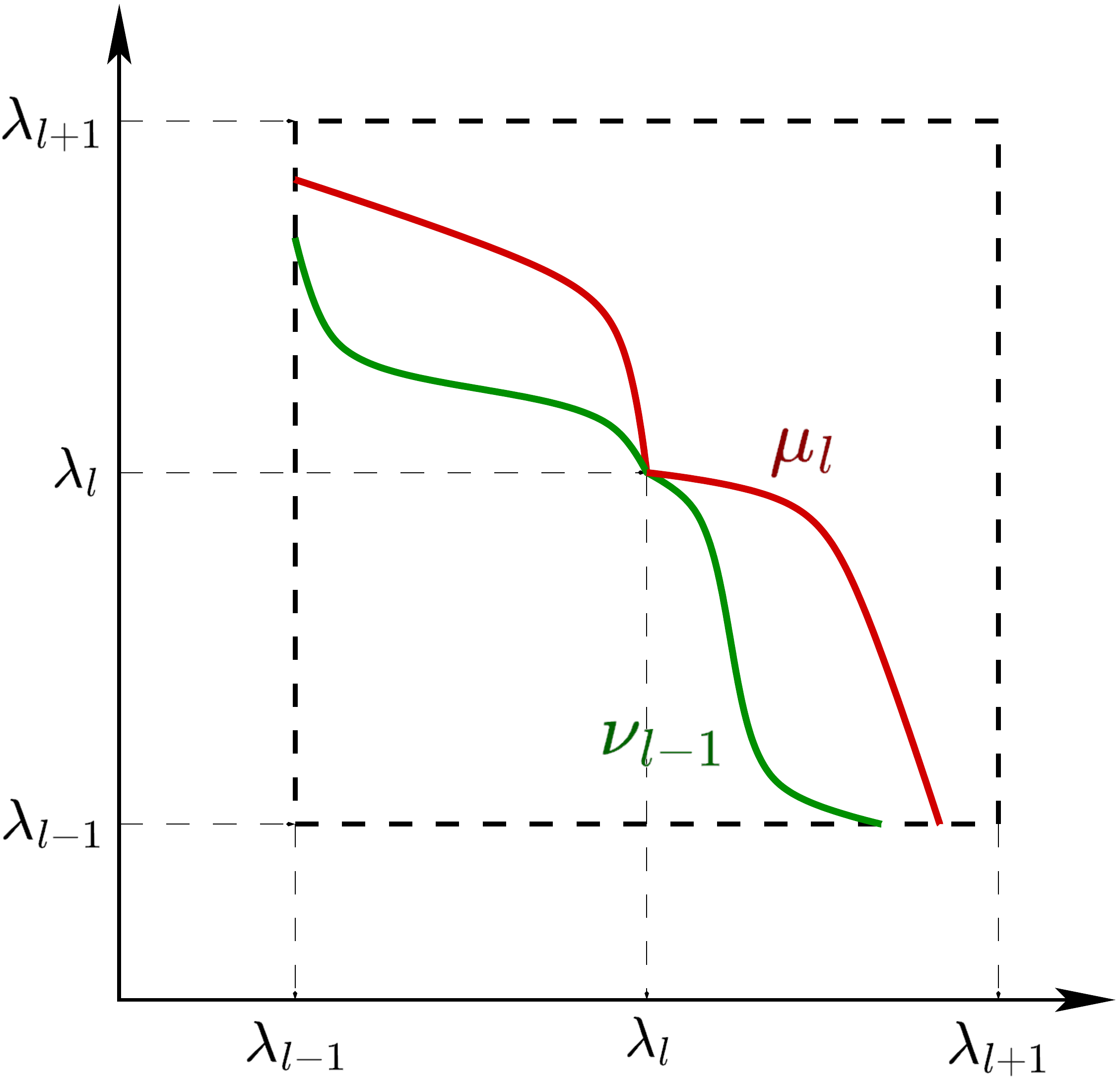}
\end{center}
\caption{Upper (in red) and lower (in green) curves of the Dancer-\Fucik spectrum.
The points of this spectrum can only lie within these two curves.}\label{FI11}
\end{figure}

For our purposes, for all~$ l \ge 2$, the region in~$Q_l$ below the lower curve of the Dancer-\Fucik spectrum
is of particular importance, since
a portion of this region contains the pairs~$(a,b)$ allowing for nontrivial solutions of~\eqref{mainab}.

To describe this portion of the plane, we define
\begin{equation}\label{best_rho}\begin{split}{\mathcal{S}}:=&\inf\Bigg\{
\mu^+(0)\,\|u\|^2_{L^2(\Omega)}
+\mu^+(1)\,\|\nabla u\|^2_{L^2(\Omega)}
\\&\qquad\qquad+
\int_{({{ 0 }},1)}\left[
{c_{N,s}} \iint_{\R^{2N}} \frac{| u(x)-u(y)|^2}{|x-y|^{N+2s}}\, dx\,dy\right]\,d\mu^+(s)
\Bigg\}.\end{split}\end{equation}
The infimum above\footnote{{F}rom now on, \label{footpahetehfh}
with a slight abuse of notation,
the quantity in brackets in formula~\eqref{best_rho} (and similar quantities)
will be abbreviated into
$$ \int_{[{{ 0 }},1]}\left[
{c_{N,s}} \iint_{\R^{2N}} \frac{| u(x)-u(y)|^2}{|x-y|^{N+2s}}\, dx\,dy\right]\,d\mu^+(s), $$
with the understanding that the measure evaluation at~$s=0$ and~$s=1$
(if nonvoid) returns the classical expressions.} is taken over all the functions~$u\in C^\infty_0(\Omega)$ satisfying~$\|u\|_{L^{2^*_{s_\sharp}}(\R^N)}=1$, with~$s_\sharp$ as in~\eqref{scritico}.
Roughly speaking, one can consider~${\mathcal{S}}$ as the analogue of the Sobolev constant
for the operator~$A_\mu$ in~\eqref{AMMU}. \medskip

For our purpose this generalized Sobolev constant is useful to identify the pairs~$(a,b)$ allowing for a nontrivial solution of~\eqref{mainab}. Specifically, these pairs are precisely the ones lying in~$Q_l$
below the lower curve of the Dancer-\Fucik spectrum and satisfying
\begin{equation}\label{sotto}
\min \set{a,b} > \lambda_l - \frac{\mathcal S}{|\Omega|^{(2 s_\sharp)/N}}.
\end{equation}
Here above and in the rest of this paper, $|\Omega|$ stands for the Lebesgue measure of~$\Omega$.
The corresponding region of interest is sketched in Figure~\ref{FI12}.

The result that we obtain is thus as follows:

\begin{theorem}\label{main1}
Let~$\mu=\mu^+-\mu^-$ with~$\mu^+$ and~$\mu^-$ satisfying~\eqref{mu00}, \eqref{mu3} and~\eqref{mu2}.

Let~$(a,b) \in Q_l$. Assume that~$b < \nu_{l-1}(a)$ and that~\eqref{sotto} is satisfied.

Then, there exists~$\gamma_0>0$, depending only on~$N$, $\Omega$, $s_\sharp$, $a$ and~$b$, such that if~$\gamma\in[0,\gamma_0]$
then problem~\eqref{mainab} admits a nontrivial solution.
\end{theorem}

\begin{figure}[h]
\begin{center}
\includegraphics[scale=.25]{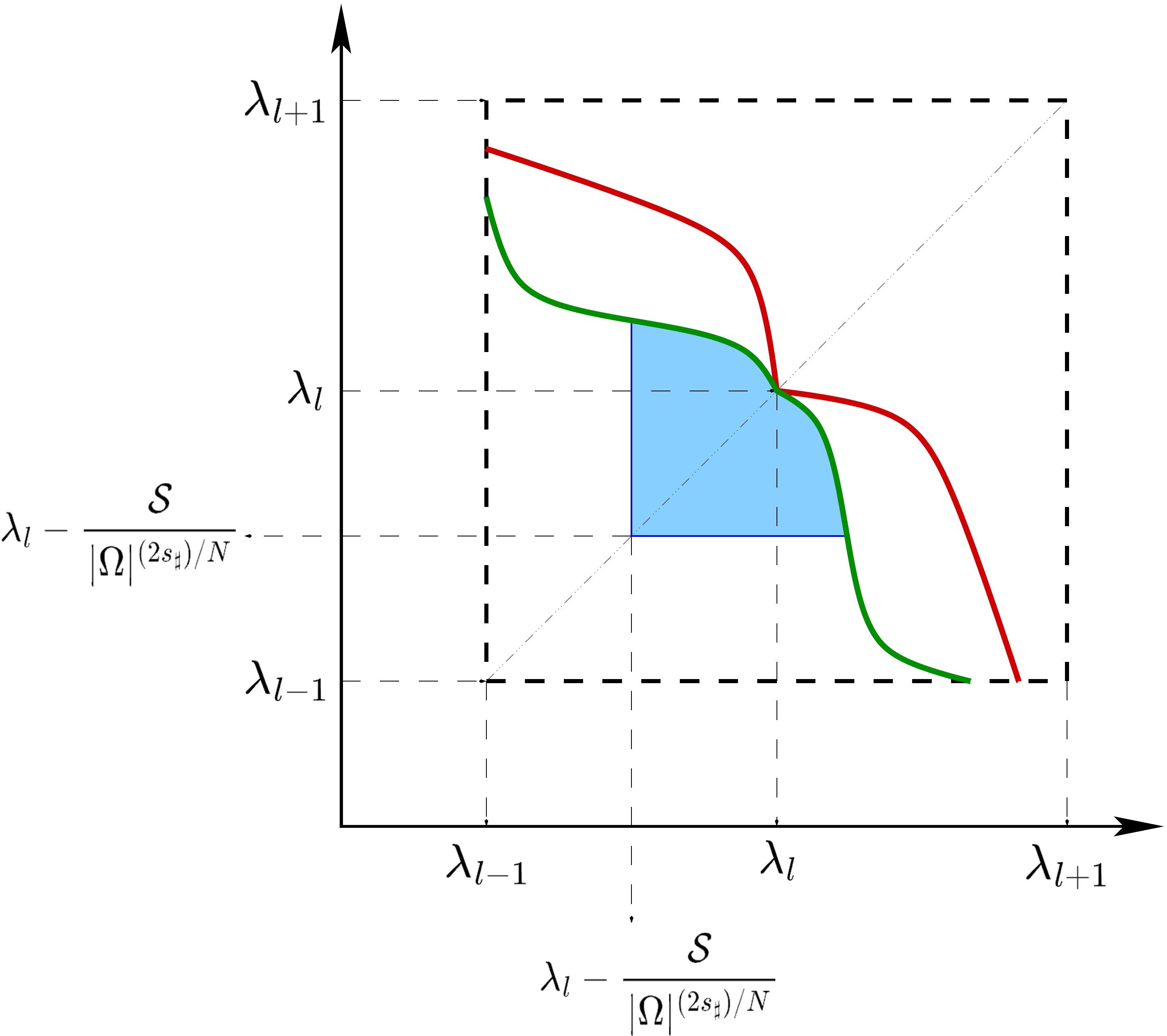}
\end{center}
\caption{The region (in light blue) below the lower curve of the Dancer-\Fucik spectrum where the existence of a nontrivial solution is guaranteed by Theorem~\ref{main1}.}\label{FI12}
\end{figure}

We stress that Theorem~\ref{main1} is not only new in its wide generality, but it
also possesses many specific cases which are also new.

In particular:
\begin{itemize}
\item If~$\mu:=\delta_1$, i.e. if~$A_\mu$ reduces to the classical Laplacian,
then problem~\eqref{mainab} has been recently studied in~\cite{MR4535437}.
In particular, \cite[Theorem 1.3]{MR4535437} provided the existence of a nontrivial solution when~$N\ge 4$.

Our result gives an existence result in a small region below the lower curve and holds for~$N\ge 3$
(hence improving the known condition on the dimension).

A detailed discussion of this type of results will be given in Corollary~\ref{ILS}.\medskip

\item If~$\mu:=\delta_s$, i.e. if the operator is the fractional Laplacian~$(-\Delta)^s$ for some~$s\in(0,1)$, our results are still new, to the best of our knowledge
(in fact, they seem to be new even in the case~$a=b$ in which no jumping nonlinearity is present,
but see~\cite{MR3060890} for related results).

We treat this case in detail in Corollary~\ref{ILFRA}.\medskip

\item If~$\mu:=\delta_1+\delta_s$, i.e. if~$A_\mu=-\Delta+(-\Delta)^s$ for some~$s\in(0,1)$,
then Theorem~\ref{main1}  is new.

In this setting, the particular case~$a=b=\lambda$ has been recently studied in~\cite[Theorem 1.4]{arXiv.2209.07502}, where a nontrivial solution was found for a suitable range of~$\lambda$.

The application of Theorem~\ref{main1} for this choice of~$\mu$ will be discussed in Corollary~\ref{mix_corollary}.\medskip

\item The case of the superposition of two nonlocal operators with different orders,
corresponding to the choice of the measure~$\mu:=\delta_{s_1}+\delta_{s_2}$ 
and to an operator of the type~$(-\Delta)^{s_1}+(-\Delta)^{s_2}$
for some~$s_1$, $s_2\in(0,1)$, is also new to our best knowledge.

\medskip

\item The case in which the measure~$\mu$ changes sign is also new to our best knowledge.
This seems to be new even in the case~$\mu:= \delta_1 -\alpha\delta_s$,
corresponding to an operator of the form~$-\Delta-\alpha(-\Delta)^s$,
where~$s\in(0,1)$ and~$\alpha$ is a small positive constant (notice the ``wrong'' sign in the second term of this operator).

A simple example in this setting is provided in Corollary~\ref{mix_signed}.
Actually, we think that our strategy on how to deal with ``wrong'' sign contributions
may be of general interest and lead to the study of a rather general class of operators with competing diffusive trends.
\medskip

\item The case of a convergent series
$$ \sum_{k=0}^{+\infty} c_k (-\Delta)^{s_k} u,\qquad{\mbox{where }}\,
\sum_{k=0}^{+\infty} c_k\in(0,+\infty),$$
with 
\begin{itemize}
\item[(i)] either~$c_k\ge0$ for all~$k\in\N$,
\item[(ii)] or
\begin{eqnarray*}& &
c_k>0\ \text{ for all } k\in\{1,\dots, \overline k\} \text{ and } \sum_{k=\overline k +1}^{+\infty} c_k \le \gamma \sum_{k=0}^{\overline k} c_k,\\
&&{\mbox{for some~$\overline k\in\N$ and~$\gamma\ge0$,}}\end{eqnarray*}
\end{itemize}
are also new (see Corollaries ~\ref{serie1} and ~\ref{serie2}).\medskip

\item The continuous superposition of fractional operators of the form
$$ \int_0^1 (-\Delta)^s u \,f(s)\,ds,$$
where~$f$ is a measurable
and non identically zero function,
is also new (see Corollary~\ref{function}).
\end{itemize}

In the forthcoming paper~\cite{CATERINA-PP}, we will also consider the case of nonlinear fractional operators of mixed order of $p$--Laplacian type.
\medskip

The rest of this paper is organized as follows. Section~\ref{op_section} gathers several estimates
of Sobolev type which will constitute the functional analytic core of our study.
Then, we present in Section~\ref{KAMqw} the variational framework in which we work and we complete the proof of the main result in Section~\ref{sec_main1}.

In Section~\ref{EXA:A} we apply the general result
in Theorem~\ref{main1} to several specific cases of interest, which are also new in the literature.

An interesting technical aspect of the proofs presented is that
our arguments are (for the first time in the literature, to our best knowledge\footnote{For instance, the measure on fractional exponents studied in~\cite{MR3485125} was supposed to be positive and supported away from~$s=0$, while we do not need either of these assumptions here.}) capable also of dealing with operators ``with the wrong sign'', i.e. the ones coming from the measure~$\mu^-$
and which have to be ``reabsorbed'' through quantitative estimates into~$\mu^+$.
We think that it is particularly remarkable that no extra assumption on the equation is needed for this. Given its interest also in practical situations (in which competing operators could participate to a complex model with opposite\footnote{For example, an interesting model of mixed operator with opposite sign could be that of a biological population in which different individuals have different dispersal behaviors. Indeed, on the one hand, in view of the L\'evy flight foraging hypothesis,
mixed diffusive operators are an interesting tool to describe searching patters of predators; on the other hand, it would be tempting to have a simplified, but effective, description of social behaviors via a reversed fractional heat equation (while diffusion tends to spread mass around, inverse diffusion favors concentration and it may therefore be a helpful tool to include animals' tendency to cluster into groups of conspecifics). In this sense, mixed operators with different signs can turn out to be handy in mathematical biology, population dynamics, social sciences, etc.} diffusion and concentration features)
we believe that this novelty can open a new direction of research and apply to other problems as well.

\section{Sobolev-type estimates}\label{op_section}
In this section, we consider a bounded open set~$\Omega\subset\R^N$
and we develop suitable energy estimates to deal with the operator in~\eqref{AMMU}.

For this, for~$s\in(0,1)$, we let
\begin{equation}\label{semisp}
[u]_{s}: = \left({c_{N,s}}\iint_{\R^{2N}} \frac{|u(x) - u(y)|^2}{|x - y|^{N+2s}}\, dx\, dy\right)^{\frac12}
\end{equation}
be the Gagliardo seminorm of a measurable function~$u : \R^N \to \R$, see e.g. \cite{MR2944369}.

The consistent choice of the normalizing constant~$c_{N,s}$ is such that
$$ \lim_{s\nearrow1}[u]_{s}=[u]_1:=\|\nabla u\|_{L^2(\R^N)}
\qquad{\mbox{and}}\qquad
\lim_{s\searrow0}[u]_s=[u]_0:=\|u\|_{L^2(\R^N)}.
$$
\medskip

We now observe that higher exponents in fractional norms control lower exponents, with uniform constants, according to the next observation:
 
\begin{lemma}\label{nons}
Let~$0\le s_1 \le s_2 \le1$. 

Then, for any measurable function~$u:\R^N\to\R$ with~$u=0$ a.e. in~$\R^N\setminus\Omega$ we have that
\begin{equation}\label{spp}
[u]_{s_1}\le c \, [u]_{s_2},
\end{equation}
for a suitable positive constant~$c=c(N,\Omega)$.
\end{lemma}

\begin{proof} First, we suppose that~$u\in C^\infty_0(\Omega)$.
A rather delicate issue is that we can take the constant~$c(N,\Omega)$ independent
of~$s_1$ and~$s_2$. To check this, it is convenient to write the Gagliardo seminorm in terms of the Fourier transform (see e.g.~\cite{MR2944369}) as
$$ [u]_s=\left(\;\int_{\R^{N}} |2\pi \xi|^{2s}\,|\widehat u(\xi)|^2\,d\xi\right)^{\frac12}.$$
We stress that, by Plancherel Theorem, this is also valid when~$s=0$ and~$s=1$.

One can combine this with the fractional Sobolev constant, which can be explicitly computed
in the Hilbert setting, see~\cite[formula~(2)]{MR3179693}, according to which
$$ [u]_s^2\ge C^\star(N,s)\,\|u\|_{L^{2^*_s}(\R^N)}^2,$$
with
\[ C^\star(N,s):= 2^{4s} \pi^{3s} \, \frac{\Gamma((N+2s)/2)}{\Gamma((N-2s)/2)} \, \left(\frac{\Gamma(N/2)}{\Gamma(N)}\right)^{2s/N},
\]
where we used the standard notation for the Gamma Function.

We recall that the Gamma Function on the real line has a minimum at~$r_\star:=1.46...$, with~$\Gamma(r_\star)>0.88$,  rising to either side of this minimum. Thus,
\[ C^\star(N,s)\ge \frac{0.88}{\Gamma(N+10)} \, \left(\frac{0.88}{\Gamma(N)}\right)^{2s/N}\ge\frac{0.88}{\Gamma(N+10)} \, \left(\frac{0.88}{\Gamma(N)}\right)^{2/N}=:C^\star(N).
\]
Consequently, using the Sobolev and H\"older inequalities,
\begin{eqnarray*}
[u]_{s_1}^2&\le& \int_{B_{1/(2\pi)}} |2\pi \xi|^{2s_1}\,|\widehat u(\xi)|^2\,d\xi+\int_{\R^N\setminus B_{1/(2\pi)}} |2\pi \xi|^{2s_2}\,|\widehat u(\xi)|^2\,d\xi\\
&\le&  \int_{\R^N} |\widehat u(\xi)|^2\,d\xi+\int_{\R^N} |2\pi \xi|^{2s_2}\,|\widehat u(\xi)|^2\,d\xi
\\&=& \|u\|^2_{L^2(\Omega)}+[u]^2_{s_2}
\\&\le& |\Omega|^{\frac{2s_2}{N}}\|u\|^2_{L^{2^*_{s_2}}(\Omega)}+[u]^2_{s_2}
\\&\le&\frac{(1+|\Omega|)^{\frac{2s_2}{N}}}{C^\star(N,s_2)}[u]^2_{s_2}+[u]^2_{s_2}
\\&\le&\frac{(1+|\Omega|)^{\frac{2}{N}}}{C^\star(N)}[u]^2_{s_2}+[u]^2_{s_2}
\end{eqnarray*}
and this proves~\eqref{spp} when~$u\in C^\infty_0(\Omega)$.

Now we perform a density argument to establish~\eqref{spp} in the general case.
To this end, let~$u:\R^N\to\R$ be a measurable function with~$u=0$ a.e. in~$\R^N\setminus\Omega$. We can assume that~$[u]_{s_2}<+\infty$, otherwise there is nothing to prove. Then, by the density of the smooth functions in the (possibly fractional) Sobolev spaces,
we find a sequence of functions~$u_k\in C^\infty_0(\Omega)$ such that~$[u_k-u]_{s_2}\to0$ as~$k\to+\infty$.

Thus, by (possibly fractional) Sobolev embeddings, up to a subsequence we can assume that~$u_k\to u$ in~$L^2(\Omega)$ and a.e. in~$\R^N$.
Accordingly, we can use the already proved version of~\eqref{spp} to infer that~$[u_k]_{s_1}\le c \, [u_k]_{s_2}$ and, as a consequence,
\begin{eqnarray*}
\liminf_{k\to+\infty}[u_k]_{s_1}\le c \liminf_{k\to+\infty} [u_k]_{s_2}\le c\left([u]_{s_2}+
\liminf_{k\to+\infty}[u_k-u]_{s_2}\right)=c\,[u]_{s_2}.
\end{eqnarray*}
Hence, the desired result in~\eqref{spp} follows 
from Fatou's Lemma.
\end{proof}

We define the space~$\mathcal{X}(\Omega)$ as the
set of measurable functions~$u:\R^N\to\R$ such that~$u=0$
in~$\R^N\setminus \Omega$ and
\begin{equation*}
\int_{[{{ 0 }}, 1]} [u]^2_{s}\, d\mu^+ (s) <+\infty.
\end{equation*}

\begin{lemma}\label{hilbert}
Suppose that~$\mu^+$ satisfies~\eqref{mu00}.
Then, $\mathcal{X}(\Omega)$ is a Hilbert space.
\end{lemma}

\begin{proof}
Let~$u_n$ be a Cauchy sequence and take~$\epsilon>0$.
Thus, there exists~$\overline{n}\in\N$ such that if~$m$, $k\ge \overline{n}$ then
\begin{equation}\label{jdiewohuggbrdbv}
\int_{[{{ 0 }}, 1]} [u_m-u_k]^2_{s}\, d\mu^+ (s)\le\epsilon.\end{equation}
Now, we distinguish two cases, either~$\mu^+(1)\neq0$
or~$\mu^+(1)=0$.

If~$\mu^+(1)\neq0$, then we obtain from~\eqref{jdiewohuggbrdbv} that
$$\epsilon\ge \mu^+(1)\int_{\Omega}|\nabla u_m-\nabla u_k|^2\,dx,$$
and therefore~$u_n$ is a Cauchy sequence in~$H^1_0(\Omega)$.
Accordingly, there exists~$u\in H^1_0(\Omega)$ such that~$u_n\to u$
in~$H^1_0(\Omega)$ as~$n\to+\infty$. 

Also, exploiting Lemma~\ref{nons}, we see that, for all~$s\in[0,1)$,
$$ [u-u_n]_{s}\le c(N,\Omega) \, [u-u_n]_{1}.$$
As a consequence,
$$ \int_{[0,1]}[u-u_n]_s^2\,d\mu^+(s)
\le c^2(N,\Omega) \mu^+([0,1]) [u-u_n]_{1}^2,$$
which gives that
$$ \lim_{n\to+\infty} \int_{[0,1]}[u-u_n]_s^2\,d\mu^+(s)=0,$$
as desired.

If instead~$\mu^+(1)=0$, then we deduce from~\eqref{jdiewohuggbrdbv}
and Lemma~\ref{nons} that
\begin{eqnarray*}&&\epsilon\ge
\int_{[{{ 0 }}, 1)} [u_m-u_k]^2_{s}\, d\mu^+ (s)\ge 
\int_{[\overline{s}, 1)} [u_m-u_k]^2_{s}\, d\mu^+ (s)\\&&\qquad
\ge \frac{1}{c^2(N,\Omega)}
\int_{[\overline{s}, 1)} [u_m-u_k]^2_{\overline{s}}\, d\mu^+ (s)
=\frac{\mu^+([\overline{s}, 1))}{c^2(N,\Omega)}\,
[u_m-u_k]^2_{\overline{s}}
.\end{eqnarray*}
We point out that~$\mu^+([\overline{s}, 1))>0$ in light of~\eqref{mu00}
and the fact that~$\mu^+(1)=0$.

Accordingly, $u_n$ is a Cauchy sequence in~$H^{\overline{s}}_0(\Omega)$
and therefore it converges to some~$u$ in~$H^{\overline{s}}_0(\Omega)$.
Hence, we can extract a subsequence~$u_{n_j}$ converging to~$u$
in~$L^2(\Omega)$ and a.e. in~$\R^N$.
Then, if~$m\ge\overline{n}$, we have that
\begin{eqnarray*}
\epsilon&\ge& \lim_{j\to+\infty}
\int_{[{{ 0 }}, 1)} [u_m-u_{n_j}]^2_{s}\, d\mu^+ (s)
\\&\ge& \lim_{j\to+\infty}
\left(\mu^+(0)\|u_m- u_{n_j}\|^2_{L^2(\Omega)}+
\int_{({{ 0 }}, 1)} [u_m-u_{n_j}]^2_{s}\, d\mu^+ (s)\right)\\
&=&\mu^+(0)\|u_m- u\|^2_{L^2(\Omega)}\\&&\qquad
+\lim_{j\to+\infty}
\int_{({{ 0 }}, 1)} \left({c_{N,s}}\iint_{\R^{2N}} \frac{|(u_m-u_{n_j})(x) - (u_m-u_{n_j})(y)|^2}{|x - y|^{N+2s}}\, dx\, dy\right) d\mu^+ (s).
\end{eqnarray*}
As a result, by Fatou's Lemma,
\begin{eqnarray*}
\epsilon&\ge&
\mu^+(0)\|u_m- u\|^2_{L^2(\Omega)}\\&&\qquad
+
\int_{({{ 0 }}, 1)} \left({c_{N,s}}\iint_{\R^{2N}}
\liminf_{j\to+\infty} \frac{|(u_m-u_{n_j})(x) - (u_m-u_{n_j})(y)|^2}{|x - y|^{N+2s}}\, dx\, dy\right) d\mu^+ (s)\\&=&
\mu^+(0)\|u_m- u\|^2_{L^2(\Omega)}\\&&\qquad
+
\int_{({{ 0 }}, 1)} \left({c_{N,s}}\iint_{\R^{2N}}
\frac{|(u_m-u)(x) - (u_m-u)(y)|^2}{|x - y|^{N+2s}}\, dx\, dy\right) d\mu^+ (s)\\
&=&
\int_{[{{ 0 }}, 1]} [u_m-u]^2_{s}\, d\mu^+ (s),
\end{eqnarray*}
which says that the sequence~$u_n$ converges to~$u$
in~${\mathcal{X}}(\Omega)$, as desired.
\end{proof}

In this setting, we can ``reabsorb'' the negative part of the signed measure~$\mu$, according to the following result:

\begin{proposition}\label{crucial} Assume~\eqref{mu3} and~\eqref{mu2}.

Then, there exists~$c_0=c_0(N,\Omega)>0$ such that,
for any~$u\in\mathcal{X}(\Omega)$, we have
\[
\int_{[{{ 0 }}, \overline s]} [u]_{s}^2 \, d\mu^- (s) \le c_0\,\gamma \int_{[\overline s, 1]} [u]^2_{s} \, d\mu(s).
\]
\end{proposition}

\begin{proof}
We notice that if~$\mu^+([\overline s,1])=0$, then condition~\eqref{mu2} would give that~$\mu^-([0,\overline s])=0$,
and therefore Proposition~\ref{crucial} would be trivially satisfied.

Thus, from now on we suppose that~$\mu^+([\overline s,1])>0$.
By applying Lemma~\ref{nons} with~$s_1: = \overline s$ and~$s_2: = s$ we infer that, for all~$s\in[\overline s,1]$,
\begin{eqnarray*}
[u]^2_{\overline s}\le c^2(N,\Omega) \, [u]^2_{s}.
\end{eqnarray*}
Similarly,  applying Lemma~\ref{nons} with~$s_1: = s$ and~$s_2: = \overline s$, for all~$s\in[0,\overline s]$ we have that
\begin{eqnarray*}
[u]^2_{s}\le c^2(N,\Omega) \, [u]^2_{\overline s}.
\end{eqnarray*}

Consequently, recalling~\eqref{mu3} and~\eqref{mu2},
\begin{eqnarray*}&&
\int_{[{{ 0 }}, \overline s]} [u]_{s}^2 \, d\mu^- (s)\le c^2(N,\Omega) \int_{[{{ 0 }}, \overline s]} [u]^2_{\overline s}\, d\mu^- (s)=c^2(N,\Omega) \,[u]^2_{\overline s}\,\mu^-\big({[{{ 0 }}, \overline s]}\big)\\&&\qquad\le
c^2(N,\Omega) \,\gamma\,[u]^2_{\overline s}\,
\mu^+\big([\overline s, 1]\big)=
c^2(N,\Omega) \,\gamma \int_{[\overline s, 1]}[u]^2_{\overline s}\,d\mu^+(s)
 \\&&\qquad\le
c^4(N,\Omega) \,\gamma \int_{[\overline s, 1]}[u]^2_{s}\,d\mu^+(s)
.
\end{eqnarray*}
This is the desired result, with~$c_0:=c^4(N,\Omega)$.
\end{proof}

\begin{proposition}\label{emb}Assume~\eqref{mu00}, \eqref{mu3} and~\eqref{mu2}.
Let~$s_\sharp\in [\overline s, 1]$ be as in~\eqref{scritico}. 

Then, there exists a positive constant~$\bar c =\bar c (N,\Omega, s_\sharp)$ such that, for any~$u\in \mathcal{X}(\Omega)$,
\begin{equation}\label{pqwodfl134rt}
[u]_{s_\sharp}\le \bar{c}\left(\, \int_{[{{ 0 }}, 1]} [u]^2_{s}\, d\mu^+ (s)\right)^{\frac12}.
\end{equation}
In particular, the space~$\mathcal{X}(\Omega)$ is continuously embedded in~$L^r(\Omega)$ for any~$r\in [1, 2^*_{s_\sharp}]$ and
compactly embedded in~$L^r(\Omega)$ for any~$r\in [1, 2^*_{s_\sharp})$.
\end{proposition}
\begin{proof} By Lemma~\ref{nons}, used here with~$s_1: = s_\sharp$ and~$s_2 := s$,
for all~$s\in[s_\sharp,1]$ we have that~$[u]_{s_\sharp}\le c(N,\Omega) \, [u]_{s}$.

As a result,
$$ \mu^+ \big([s_\sharp, 1]\big)\,[u]_{s_\sharp}^2\le c^2(N,\Omega) \, \int_{[s_\sharp, 1]}[u]_{s}^2\,d\mu^+(s)\le
c^2(N,\Omega)\int_{[{{ 0 }}, 1]} [u]^2_{s}\, d\mu^+ (s).
$$
This and~\eqref{scritico} yield the desired result.\end{proof}

\section{Variational setting}\label{KAMqw}
In this section, we cast problem~\eqref{mainab} into a suitable variational setting.
To start, we state the following definition.
\begin{definition}\label{wsol2}
A weak solution of problem~\eqref{mainab} is a function~$u \in \mathcal{X}(\Omega)$ such that, for all~$ v \in \mathcal{X}(\Omega)$,
\[
\begin{split}
&\int_{[{{ 0 }}, 1]}\left(c_{N,s}\iint_{\R^{2N}} \frac{(u(x) - u(y))\, (v(x) - v(y))}{|x - y|^{N+2s}}\, dx\, dy\right) d\mu^+(s)\\
&\qquad- \int_{[{{ 0 }}, \overline s]}\left(c_{N,s}\iint_{\R^{2N}} \frac{(u(x) - u(y))\, (v(x) - v(y))}{|x - y|^{N+2s}}\, dx\, dy\right) d\mu^-(s)\\
&\qquad\qquad= \int_{\Omega} (bu^+ - a u^-) v\, dx + \int_\Omega |u|^{2_{s_\sharp}^\ast - 2}\, uv\, dx.
\end{split}
\]
\end{definition}

The variational functional~$E: \mathcal{X}(\Omega)\to \R$ associated with problem~\eqref{mainab} is defined by
\begin{equation}\label{fun}
E(u) = \half \,\int_{[{{ 0 }}, 1]} [u]^2_{s}\, d\mu^+ (s)
 -\half\int_{[{{ 0 }}, \overline s]} [u]^2_s\, d\mu^-(s) - \half \int_\Omega \left[a\, (u^-)^2 + b\, (u^+)^2\right]\, dx - \frac{1}{2_{s_\sharp}^*} \int_\Omega |u|^{2_{s_\sharp}^*}\, dx.
\end{equation}
\begin{remark}
Note that in the functional~\eqref{fun} the term arising from the negative part of the measure~$\mu$ can be absorbed in the norm. In fact, by Proposition~\ref{crucial} we have that
\[
\int_{[{{ 0 }}, \overline s]} [u]_s^2 \, d\mu^- (s) \le c_0(N,\Omega) \,\gamma\int_{[\overline s, 1]} [u]^2_s \, d\mu(s) \le c_0(N,\Omega) \,\gamma\int_{[{{ 0 }}, 1]} [u]^2_{s}\, d\mu^+ (s).
\]
In particular, if~$\gamma$ is sufficiently small (possibly depending on~$N$ and~$\Omega$) it follows that
\[
E(u) \ge \frac{1}{4}\int_{[{{ 0 }}, 1]} [u]^2_{s}\, d\mu^+ (s)
- \half \int_\Omega \left[a\, (u^-)^2 + b\, (u^+)^2\right] \,dx - \frac{1}{2_{s_\sharp}^*} \int_\Omega |u|^{2_{s_\sharp}^*}\, dx.
\]
\end{remark}

We now state a weak convergence result (to be used below in the analysis of
Palais-Smale sequences in the forthcoming Proposition~\ref{P51}):

\begin{lemma}\label{Vitali}
Let~$u_n$ be a bounded sequence in~$\mathcal X(\Omega)$. 

Then, there exists~$u:\R^N\to\R$ such that, up to a subsequence, for any~$v\in\mathcal X(\Omega)$,
\begin{equation}\label{Vconv}
\begin{split}
&\lim_{n\to+\infty}\int_{[{{ 0 }}, \overline s]} \left(\;\iint_{\R^{2N}} \frac{c_{N,s}(u_n(x)-u_n(y)) (v(x)-v(y))}{|x-y|^{N+2s}} \, dx\, dy\right)\, d\mu^-(s)  \\
&\qquad =\int_{[{{ 0 }}, \overline s]} \left(\;\iint_{\R^{2N}} \frac{c_{N,s}(u(x)-u(y)) (v(x)-v(y))}{|x-y|^{N+2s}} \, dx\, dy\right)\, d\mu^-(s).
\end{split}
\end{equation}
Also, 
\begin{equation}\label{lunoconv}
{\mbox{$u_n$ converges to~$u$ in~$L^1(\Omega)$ as~$n\to+\infty$.}}
\end{equation}
\end{lemma}

\begin{proof} By~\eqref{semisp} and Proposition~\ref{crucial},
\begin{eqnarray*}&&
\int_{[{{ 0 }}, \overline s]}\left(
{c_{N,s}} \iint_{\R^{2N}} \frac{|u_n(x) - u_n(y)|^2}{|x - y|^{N+2s}}\, dx\, dy\right)\,d\mu^-(s)
=\int_{[{{ 0 }}, \overline s]} [u_n]_{s}^2 \, d\mu^- (s)\\&&\qquad\qquad \le c_0\,\gamma \int_{[\overline s, 1]} [u_n]^2_{s} \, d\mu(s)\le c_0\,\gamma \int_{[{{ 0 }}, 1]} [u_n]^2_{s}\, d\mu^+ (s),
\end{eqnarray*}
which is bounded uniformly in~$n$.

The desired result in~\eqref{Vconv}
now follows from the Banach-Alaoglu Theorem.

It remains to establish~\eqref{lunoconv}. To this aim, we consider three cases:
$\mu^-$ is the zero measure, $\mu^-$ is a Dirac measure at~$0$
and, as a last possibility, $\mu^-((0,\overline{s}])>0$.

In the first two cases, we could have used directly Proposition~\ref{emb}, obtaining that, up to a subsequence, $u_n$ converges to some~$u$
in~$L^2(\R^N)$. This~$u$ would have satisfied 
both~\eqref{Vconv} and~\eqref{lunoconv} (because in the first case~\eqref{Vconv} is void and in the second case,
recalling footnote~\ref{footpahetehfh} on page~\pageref{footpahetehfh}, it has to be interpreted as
a weak convergence in~$L^2(\R^N)$). 

So, we can focus on the third case, namely we suppose that~$\mu^-((0,\overline{s}])>0$. Hence, by the Dominated Convergence Theorem, 
$$ \lim_{\epsilon\searrow0} \mu^-([\epsilon,\overline{s}])
=\mu^-((0,\overline{s}])>0.$$
Accordingly, there exists~$\epsilon_0\in(0,\overline{s}]$ such that~$\mu^-([\epsilon_0,\overline{s}])>0$.

{F}rom this and Lemma~\ref{nons},
we have that
$$\mu^-([\epsilon_0,\overline{s}])[u_n]^2_{\epsilon_0}\le c^2(N,\Omega) \, \int_{[\epsilon_0,\overline{s}]}[u_n]_{s}^2\,d\mu^-(s),
$$
which is bounded uniformly in~$n$, thanks to Proposition~\ref{crucial}.
Thus, by the compactness result for fractional Sobolev spaces (see e.g.~\cite{MR2944369}), we obtain~\eqref{lunoconv}, as desired. 
\end{proof}

Now we address the convergence of the Palais-Smale sequences. For this, we first point out that, in view of~\eqref{sotto},
there exists~$\epsilon_0=\epsilon_0(N,\Omega, s_\sharp, a,b)\in\left(0,\frac{{\mathcal{S}}}{|\Omega|^{(2 s_\sharp)/N}}\right)$ such that
\begin{equation*}
\min \set{a,b} > \lambda_l - \frac{\mathcal S}{|\Omega|^{(2 s_\sharp)/N}}+\epsilon_0.
\end{equation*}
Hence, we define
$$ \theta_0=\theta_0(N,\Omega, s_\sharp, a,b):=\frac{|\Omega|^{(2 s_\sharp)/N}}{\mathcal S}\,\epsilon_0\in(0,1)$$
and we see that
\begin{equation}\label{sottoBIS}
\min \set{a,b} > \lambda_l - \frac{\mathcal S}{|\Omega|^{(2 s_\sharp)/N}}(1-\theta_0).
\end{equation}

With this notation, we have:

\begin{proposition}\label{P51}
Let~$\mathcal S$ be as in~\eqref{best_rho} and
\begin{equation}\label{CASTA}
c^\ast:= \frac{s_\sharp}{N}\, \big((1-\theta_0)\mathcal S \big)^{\frac{N}{2 s_\sharp}}.
\end{equation}

Then, there exists~$\gamma_0>0$, depending on~$N$, $\Omega$,
$s_\sharp$, $a$ and~$b$, such that if~$\gamma\in[0,\gamma_0]$
and~$c \in (0,c^\ast)$,
then every~\PS{c} sequence of the functional~\eqref{fun} has a subsequence that converges weakly to a nontrivial critical point of~\eqref{fun}.
\end{proposition}

\begin{proof}
Let~$u_n$ be a~\PS{c} sequence of the functional~$E$, i.e.
\begin{equation}\label{c1}
\begin{split}&
\lim_{n\to+\infty}
E(u_n) \\
= \;&\lim_{n\to+\infty}\half  \int_{[{{ 0 }}, 1]} [u_n]^2_{s}\, d\mu^+ (s)
-\half\int_{[{{ 0 }}, \overline s]} [u_n]^2_s \, d\mu^-(s)
\\&\qquad - \half \int_\Omega \left[a\, (u_n^-)^2 + b\, (u_n^+)^2\right] \,dx - \frac{1}{2_{s_\sharp}^*} \int_\Omega |u_n|^{2_{s_\sharp}^*}\, dx\\= \;&c
\end{split}
\end{equation}
and~$ dE (u_n)$ converges to~$0$ in the dual of~$\mathcal{X}(\Omega)$,
namely
\begin{equation}\label{c2}
\lim_{n\to+\infty}\sup_{ v \in \mathcal{X}(\Omega)}\big|\langle dE (u_n), v\rangle\big| =0.
\end{equation}
Since, for all~$v \in \mathcal{X}(\Omega)$,
\begin{equation*}\begin{split}
\langle dE (u_n), v\rangle =
\;& \int_{[{{ 0 }}, 1]}\left(c_{N,s}
\iint_{\R^{2N}} \frac{(u_n(x) - u_n(y))\, (v(x) - v(y))}{|x - y|^{N+2s}}\, dx\, dy\right) d\mu^+(s)\\
&\qquad- \int_{[{{ 0 }}, \overline s]}\left(c_{N,s}
\iint_{\R^{2N}} \frac{(u_n(x) - u_n(y))\, (v(x) - v(y))}{|x - y|^{N+2s}}\, dx\, dy\right) d\mu^-(s)\\
&\qquad -\int_\Omega \big(b\, u_n^+ -a\, u_n^-\big)v \, dx - \int_\Omega |u_n|^{2_{s_\sharp}^\ast - 2}\, u_n v\, dx,
\end{split}
\end{equation*}
choosing~$v:=u_n$ in~\eqref{c2},
we obtain that
\begin{equation}\label{hufewghwutgui0987654}
\begin{split}0=\;& \lim_{n\to+\infty}
\langle dE (u_n), u_n\rangle\\=
\;& \lim_{n\to+\infty}\int_{[{{ 0 }}, 1]}\left(c_{N,s}
\iint_{\R^{2N}} \frac{|u_n(x) - u_n(y)|^2}{|x - y|^{N+2s}}\, dx\, dy\right) d\mu^+(s)\\
&\qquad- \int_{[{{ 0 }}, \overline s]}\left(c_{N,s}
\iint_{\R^{2N}} \frac{|u_n(x) - u_n(y)|^2}{|x - y|^{N+2s}}\, dx\, dy\right) d\mu^-(s)\\
&\qquad -\int_\Omega \big[b\, (u_n^+)^2 +a\, (u_n^-)^2\big]\, dx - \int_\Omega |u_n|^{2_{s_\sharp}^\ast}\, dx\\=\;&  \lim_{n\to+\infty}
\int_{[{{ 0 }}, 1]} [u_n]^2_{s}\, d\mu^+ (s) -\int_{[{{ 0 }}, \overline s]} [u_n]^2_s \, d\mu^-(s) - \int_\Omega \left[a\, (u_n^-)^2 + b\, (u_n^+)^2\right] \,dx - \int_\Omega |u_n|^{2_{s_\sharp}^*}\, dx
\\=\;&  \lim_{n\to+\infty} 2E(u_n)+\left(\frac2{2^*_{s_\sharp}}-1\right)\int_\Omega |u_n|^{2_{s_\sharp}^*}\, dx.
\end{split}
\end{equation}
Combining this and~\eqref{c1}, we infer that
\begin{equation}\label{jdewfguewogewuootr6u768i0}
0=2c+
\left(\frac2{2^*_{s_\sharp}}-1\right)\lim_{n\to+\infty}\int_\Omega |u_n|^{2_{s_\sharp}^*}\, dx,\end{equation}
yielding that, for large~$n$,
\begin{equation}\label{prima}
\left(\half -  \frac{1}{2_{s_\sharp}^*}\right) \|u_n\|^{2^*_{s_\sharp}}_{L^{2^*_{s_\sharp}}(\Omega)} \le c+1. 
\end{equation}
Moreover, from this and the H\"older inequality, it follows that
\begin{equation}\label{prima22}\begin{split}&
\int_\Omega \left[a\, (u_n^-)^2 + b\, (u_n^+)^2\right]\, dx \le \max\{ a, b\}\|u_n\|^2_{L^2(\Omega)} \le \max\{ a, b\} |\Omega|^{(2 s_\sharp)/N} \|u_n\|^2_{L^{2^*_{s_\sharp}}(\Omega)}\\
&\qquad\le \left(\frac{(c+1)N}{s_\sharp}\right)^{\frac2{2^*_{s_\sharp}}} \max\{ a, b\} |\Omega|^{(2 s_\sharp)/N}.
\end{split}\end{equation}

Now, by~\eqref{c1} and
Proposition~\ref{crucial}, we have that, as soon as~$n$ is big enough,
\begin{eqnarray*}
&&\frac12\int_{[{{ 0 }}, 1]} [u_n]^2_{s}\, d\mu^+ (s)\\
&\le&\half\int_{[{{ 0 }}, \overline s]} [u_n]^2_s \, d\mu^-(s) +\half \int_\Omega \left[a\, (u_n^-)^2 + b\, (u_n^+)^2\right] \,dx +\frac{1}{2_{s_\sharp}^*} \int_\Omega |u_n|^{2_{s_\sharp}^*}\, dx+c+1\\&\le&
\frac{c_0(N,\Omega)\gamma}2\int_{[ \overline s,1]} [u_n]^2_s \, d\mu(s) +\half \int_\Omega \left[a\, (u_n^-)^2 + b\, (u_n^+)^2\right] \,dx +\frac{1}{2_{s_\sharp}^*} \int_\Omega |u_n|^{2_{s_\sharp}^*}\, dx+c+1,
\end{eqnarray*}
and therefore, if~$\gamma$ is sufficiently small (possibly depending on~$N$ and~$\Omega$),
$$\frac14 \int_{[{{ 0 }}, 1]} [u_n]^2_{s}\, d\mu^+ (s)\le
\half \int_\Omega \left[a\, (u_n^-)^2 + b\, (u_n^+)^2\right] \,dx +\frac{1}{2_{s_\sharp}^*} \int_\Omega |u_n|^{2_{s_\sharp}^*}\, dx+c+1.
$$
{F}rom this, \eqref{prima} and~\eqref{prima22}, we obtain that
$$
\frac14 \int_{[{{ 0 }}, 1]} [u_n]^2_{s}\, d\mu^+ (s)\le
\half \left(\frac{(c+1)N}{s_\sharp}\right)^{\frac2{2^*_{s_\sharp}}} \max\{ a, b\} |\Omega|^{(2 s_\sharp)/N} +\frac{N(c+1)}{2s_\sharp},
$$
which says that~$ \int_{[{{ 0 }}, 1]} [u_n]^2_{s}\, d\mu^+ (s)$
is uniformly bounded in~$n$.

Hence, in view of Lemma~\ref{hilbert} and
Proposition~\ref{emb}, there exists~$u\in \mathcal{X}(\Omega)$ such that, up to subsequences,
\begin{equation}\label{weakun}\begin{split}
u_n\rightharpoonup u &\text{ in } \mathcal{X}(\Omega),\\ 
u_n\to u &\text{ in } L^{r}(\Omega) \text{ for every } r\in [1, 2_{s_\sharp}^*),\\ 
u_n\to u &\text{ a.e. in } \Omega.
\end{split}\end{equation}
Furthermore, we observe that~$u$ is a weak solution of~\eqref{mainab},
according to Definition~\ref{wsol2}, thanks to the convergence statements
in~\eqref{weakun} and Lemma~\ref{Vitali}.

It remains to prove that
\begin{equation}\label{uzerosper00}
u\not\equiv 0.\end{equation}
To this end,
suppose by contradiction that~$u\equiv 0$. 
We recall from~\eqref{hufewghwutgui0987654} that
\begin{eqnarray*} 0&=&
\lim_{n\to+\infty}\langle dE(u_n), u_n\rangle \\
&=&\lim_{n\to+\infty} \int_{[{{ 0 }}, 1]} [u_n]^2_{s}\, d\mu^+ (s) -\int_{[{{ 0 }}, \overline s]} [u_n]^2_s \, d\mu^-(s) - \int_\Omega \left[a\, (u_n^-)^2 + b\, (u_n^+)^2\right] \,dx - \int_\Omega |u_n|^{2_{s_\sharp}^\ast}\, dx \\&=&
\lim_{n\to+\infty} \int_{[{{ 0 }}, 1]} [u_n]^2_{s}\, d\mu^+ (s) -\int_{[{{ 0 }}, \overline s]} [u_n]^2_s \, d\mu^-(s) - \int_\Omega |u_n|^{2_{s_\sharp}^\ast}\, dx.
\end{eqnarray*}
Thus, exploiting Proposition~\ref{crucial}, we have that
\begin{eqnarray*} 0&\ge&
\lim_{n\to+\infty} \big(1-c_0(N,\Omega)\gamma\big)\int_{[{{ 0 }}, 1]} [u_n]^2_{s}\, d\mu^+ (s)  - \int_\Omega |u_n|^{2_{s_\sharp}^\ast}\, dx.
\end{eqnarray*}

Accordingly, recalling the definition of~${\mathcal{S}}$ in~\eqref{best_rho}, we infer that
\begin{eqnarray*}
0 &\ge& \lim_{n\to+\infty} 
\big(1-c_0(N,\Omega)\gamma\big)\int_{[{{ 0 }}, 1]} [u_n]^2_{s}\, d\mu^+ (s)
-\mathcal{S}^{-\frac{2_{s_\sharp}^*}2}
\left(\, \int_{[{{ 0 }}, 1]} [u_n]^2_{s}\, d\mu^+ (s)\right)^{\frac{2_{s_\sharp}^*}2}\\
&=&  \big(1-c_0(N,\Omega)\gamma\big)\\&&\qquad\times\lim_{n\to+\infty} 
\int_{[{{ 0 }}, 1]} [u_n]^2_{s}\, d\mu^+ (s) \left(1 - \frac{\mathcal{S}^{-\frac{2_{s_\sharp}^*}2}}{\big(1-c_0(N,\Omega)\gamma\big)}
\left(\, \int_{[{{ 0 }}, 1]} [u_n]^2_{s}\, d\mu^+ (s)\right)^{\frac{2_{s_\sharp}^*}2-1 }\right).
\end{eqnarray*}

Now, choosing~$\gamma$ sufficiently small (possibly in dependence of~$N$
and~$\Omega$) so that~$1-c_0(N,\Omega)\gamma>0$, we conclude that
\begin{equation}\label{jdiweogbfgsdkgjk00}
0\ge \lim_{n\to+\infty} \int_{[{{ 0 }}, 1]} [u_n]^2_{s}\, d\mu^+ (s)
\left(1 - \frac{\mathcal{S}^{-\frac{2_{s_\sharp}^*}2}}{\big(1-c_0(N,\Omega)\gamma\big)}
\left(\, \int_{[{{ 0 }}, 1]} [u_n]^2_{s}\, d\mu^+ (s)\right)^{\frac{2_{s_\sharp}^*}2-1 }\right).
\end{equation}

We observe that
\begin{equation}\label{jdiweogbfgsdkgjk}
\liminf_{n\to+\infty} \int_{[{{ 0 }}, 1]} [u_n]^2_{s}\, d\mu^+ (s)>0.
\end{equation}
Indeed, suppose by contradiction that
\begin{equation*}
\liminf_{n\to+\infty} \int_{[{{ 0 }}, 1]} [u_n]^2_{s}\, d\mu^+ (s)=0.
\end{equation*}
Then, by Proposition~\ref{crucial} we would also have that
\begin{equation*}
\liminf_{n\to+\infty} \int_{[{{ 0 }}, 1]} [u_n]^2_{s}\, d\mu^- (s)=0,
\end{equation*}
and therefore, by~\eqref{c1},
\begin{equation*}
0<c=\liminf_{n\to+\infty}\left(-\frac{1}{2_{s_\sharp}^*} \int_\Omega |u_n|^{2_{s_\sharp}^*}\, dx\right)\le 0,
\end{equation*}
which is a contradiction and thus establishes~\eqref{jdiweogbfgsdkgjk}.

Thanks to~\eqref{jdiweogbfgsdkgjk00} and~\eqref{jdiweogbfgsdkgjk},
we conclude that
\begin{equation*}
0\ge \limsup_{n\to+\infty} 
\left(1 - \frac{\mathcal{S}^{-\frac{2_{s_\sharp}^*}2}}{\big(1-c_0(N,\Omega)\gamma\big)}
\left(\, \int_{[{{ 0 }}, 1]} [u_n]^2_{s}\, d\mu^+ (s)\right)^{\frac{2_{s_\sharp}^*}2-1 }\right),
\end{equation*}
which in turn gives that
\begin{equation}\label{dhuwegi4utg4u3gtu43}
\liminf_{n\to+\infty}\int_{[{{ 0 }}, 1]} [u_n]^2_{s}\, d\mu^+ (s)\ge\big(1-c_0(N,\Omega)\gamma\big)^{\frac2{2_{s_\sharp}^*-2} }
\mathcal{S}^{\frac{2_{s_\sharp}^*}{2_{s_\sharp}^*-2}}
=\big(1-c_0(N,\Omega)\gamma\big)^{\frac{N-2s_\sharp}{2s_\sharp} }
\mathcal{S}^{\frac{N}{2_{s_\sharp}}}.
\end{equation}

Additionally, using again~\eqref{c1}, and recalling the strong convergence
statement in~\eqref{weakun},
\begin{equation*}c=
\lim_{n\to+\infty}\half  \int_{[{{ 0 }}, 1]} [u_n]^2_{s}\, d\mu^+ (s)
-\half\int_{[{{ 0 }}, \overline s]} [u_n]^2_s \, d\mu^-(s)- \frac{1}{2_{s_\sharp}^*} \int_\Omega |u_n|^{2_{s_\sharp}^*}\, dx.
\end{equation*}
Hence, exploiting Proposition~\ref{crucial}, this gives that
\begin{equation*}c\ge
\lim_{n\to+\infty}\frac{1-c_0(N,\Omega)\gamma}2
\int_{[{{ 0 }}, 1]} [u_n]^2_{s}\, d\mu^+ (s)
- \frac{1}{2_{s_\sharp}^*} \int_\Omega |u_n|^{2_{s_\sharp}^*}\, dx.
\end{equation*}
{F}rom this and~\eqref{jdewfguewogewuootr6u768i0} it follows that
\begin{equation*}c\ge
\lim_{n\to+\infty}\frac{1-c_0(N,\Omega)\gamma}2
\int_{[{{ 0 }}, 1]} [u_n]^2_{s}\, d\mu^+ (s)
- \frac{(N-2s_\sharp)c}{2{s_\sharp}},
\end{equation*}
and therefore
\begin{equation*}\frac{Nc}{2s_\sharp}\ge
\lim_{n\to+\infty}\frac{1-c_0(N,\Omega)\gamma}2
\int_{[{{ 0 }}, 1]} [u_n]^2_{s}\, d\mu^+ (s).
\end{equation*}
This and~\eqref{dhuwegi4utg4u3gtu43} give that
\begin{equation*}\frac{Nc}{s_\sharp}\ge
\big(1-c_0(N,\Omega)\gamma\big)^{\frac{N}{2s_\sharp}}
\mathcal{S}^{\frac{N}{2_{s_\sharp}}},
\end{equation*}
and thus, recalling the definition of~$c^\ast$ in~\eqref{CASTA},
\begin{equation*}c^\ast>c\ge
\big(1-c_0(N,\Omega)\gamma\big)^{\frac{N}{2s_\sharp}}\frac{s_\sharp}{N}
\mathcal{S}^{\frac{N}{2_{s_\sharp}}}
=\left(\frac{1-c_0(N,\Omega)\gamma}{1-\theta_0}\right)^{\frac{N}{2s_\sharp}} c^\ast.
\end{equation*}

We point out that this implies that
$$ c_0(N,\Omega)\gamma\ge \theta_0,$$
hence, choosing~$\gamma$ sufficiently small, possibly in dependence of~$N$, $\Omega$, $s_\sharp$, $a$ and~$b$,
we obtain the desired contradiction.

Then, the claim in~\eqref{uzerosper00} is established, completing the
proof of Proposition~\ref{P51}.
\end{proof}

\section{Existence theory and proof of Theorem~\ref{main1}}\label{sec_main1}

With the preliminary work carried out so far, we are now in position of proving the existence
result in Theorem~\ref{main1}.

\begin{proof}[Proof of Theorem~\ref{main1}]
The aim is to exploit~\cite[Theorem 4.1]{MR4535437}. To this end, we set~$E_l$ to be the eigenspace associated with the eigenvalue~$\lambda_l$
and we remark that, for all~$u\in E_j$ with~$j\in\{1,\dots,l\}$,
$$ \int_{[{{ 0 }}, 1]} [u]^2_{s}\, d\mu (s) =\lambda_j \|u\|^2_{L^2(\Omega)}\le \lambda_l \|u\|^2_{L^2(\Omega)}.$$

Now, we observe that, by the H\"older inequality,
\[
\|u\|^{2^*_{s_\sharp}}_{L^2(\Omega)}\le |\Omega|^{\frac{2 s_\sharp}{N-2 s_\sharp}} \|u\|_{L^{2_{s_\sharp}^*}(\Omega)}^{2^*_{s_\sharp}},
\]
and thus
\[ \int_\Omega |u|^{2_{s_\sharp}^*}\, dx =  \|u\|_{L^{2_{s_\sharp}^*}(\Omega)}^{2^*_{s_\sharp}}\ge  |\Omega|^{-\frac{2 s_\sharp}{N-2s_\sharp}} \|u\|^{2^*_{s_\sharp}}_{L^2(\Omega)}.
\]
Consequently, recalling the definition of the functional in~\eqref{fun}, we have that, for all~$u\in E_j$ with~$j\in\{1,\dots,l\}$,
\begin{equation}\label{diywtr843yt9876543}\begin{split}
&E(u)\\
=\;&
\half \,\int_{[{{ 0 }}, 1]} [u]^2_{s}\, d\mu^+ (s)
 -\half\int_{[{{ 0 }}, \overline s]} [u]^2_s\, d\mu^-(s) - \half \int_\Omega \left[a\, (u^-)^2 + b\, (u^+)^2\right]\, dx - \frac{1}{2_{s_\sharp}^*} \int_\Omega |u|^{2_{s_\sharp}^*}\, dx
\\
\le\;&\half \,\int_{[{{ 0 }}, 1]} [u]^2_{s}\, d\mu (s)
- \frac{ \min \set{a,b}}2\|u\|_{L^2(\Omega)}^2 - \frac{1}{2_{s_\sharp}^*} |\Omega|^{-\frac{2 s_\sharp}{N-2s_\sharp}} \|u\|^{2^*_{s_\sharp}}_{L^2(\Omega)}\\
\le\;&\frac{\lambda_l}2 \|u\|^2_{L^2(\Omega)}
- \frac{ \min \set{a,b}}2\|u\|_{L^2(\Omega)}^2 - \frac{1}{2_{s_\sharp}^*} |\Omega|^{-\frac{2 s_\sharp}{N-2s_\sharp}} \|u\|^{2^*_{s_\sharp}}_{L^2(\Omega)}.
\end{split}\end{equation}

Now we consider the function
$$ h(t):=\frac12\big( \lambda_l
- \min \set{a,b}\big)t^2 - \frac{1}{2_{s_\sharp}^*} |\Omega|^{-\frac{2 s_\sharp}{N-2s_\sharp}}t^{2^*_{s_\sharp}} $$
and we observe that~$\lambda_l- \min \set{a,b}>0$. Accordingly, we obtain that
$$ \max_{t\ge0} h(t) =\left(\frac12-\frac1{2_{s_\sharp}^*}\right) |\Omega| \big(\lambda_l
- \min \set{a,b}\big)^{\frac{N}{2s_\sharp}}.$$

Plugging this information into~\eqref{diywtr843yt9876543}, we conclude that
$$ E(u)\le\left(\frac12-\frac1{2_{s_\sharp}^*}\right) |\Omega| \big(\lambda_l
- \min \set{a,b}\big)^{\frac{N}{2s_\sharp}}.
$$
Therefore, exploiting the assumption in~\eqref{sottoBIS} and recalling the definition of~$c^\ast$ in~\eqref{CASTA}, we obtain that,
for all~$u\in E_j$ with~$j\in\{1,\dots,l\}$,
$$ E(u)< \left(\frac12-\frac1{2_{s_\sharp}^*}\right) |\Omega| \left(\frac{(1-\theta_0)\mathcal S}{|\Omega|^{(2 s_\sharp)/N}}\right)^{\frac{N}{2s_\sharp}}
=\left(\frac12-\frac1{2_{s_\sharp}^*}\right) \big( (1-\theta_0){\mathcal S}\big)^{\frac{N}{2s_\sharp}}=c^\ast.
$$

Thus, Propositions~\ref{P51} ensures the convergence of Palais-Smale sequences below the threshold~$c^\ast$, 
provided that~$\gamma$ is sufficiently small.
This allows us to use~\cite[Theorem 4.1]{MR4535437},
from which the desired result follows. 
\end{proof}

\section{Examples and applications}\label{EXA:A}

Note that the operator introduced in~\eqref{mainab} is very general and we can employ it to produce a wide number of new interesting existence results for critical problems, depending on the particular choice of the measure~$\mu$. We showcase some of these cases here below.\medskip

We start proving that, by choosing~$\mu$ in a proper way, 
our result can be compared to~\cite[Theorem 1.3]{MR4535437} and~\cite[Theorem 1.2]{MPSS}.
\begin{corollary}\label{ILS}
Let~$\lambda_l$ be the sequence of Dirichlet eigenvalues of~$-\Delta$
and~$2^* = 2N/(N-2)$ be the classical Sobolev exponent. 

If~$(a, b)\in Q_l$, $b <\nu_{l-1}(a)$ and
\begin{equation*}
\min \set{a,b} > \lambda_l -\frac{S}{|\Omega|^{2/N}}
\end{equation*}
where~$S$ as denotes the classical best Sobolev constant, then
problem
\begin{equation*}
\left\{\begin{aligned}
-\Delta \, u & = bu^+ - au^- + |u|^{2^* - 2}\, u && \text{in } \Omega,\\
u & = 0 && \text{in } \ \partial\Omega,
\end{aligned}\right.
\end{equation*}
possesses a nontrivial solution.
\end{corollary}
\begin{proof}
Let~$\mu:= \delta_1$ be the Dirac measure centred at the point~$1$. 
In this case~$\mu^-=0$ as well and~$\mu$ satisfies~\eqref{mu00}, \eqref{mu3}
and~\eqref{mu2}. 
Furthermore, we can take~$\overline s:=1$ and~$s_\sharp:=1$, so that~${\mathcal{S}}$ in~\eqref{best_rho} reduces to the classical Sobolev constant. The desired result now follows from
Theorem~\ref{main1}.
\end{proof}

To our best knowledge, our main result is new even for the case of the fractional Laplacian, and even for the case~$a=b$ in which the jumping nonlinearity is not present. For the reader convenience, we state this result here below:

\begin{corollary}\label{ILFRA}
Let~$s\in [{{ 0 }}, 1)$  
and~$2^*_s = (2N)/(N-2s)$ be the critical fractional Sobolev exponent. 

Denote by~$\lambda_l$ the sequence of Dirichlet eigenvalues of~$(-\Delta)^s$ and by~$S(s)$ 
the fractional Sobolev constant corresponding to~$(-\Delta)^s$.

If~$(a, b)\in Q_l$, $b <\nu_{l-1}(a)$ and
\begin{equation*}
\min \set{a,b} > \lambda_l - \frac{S(s)}{|\Omega|^{2s/N}},
\end{equation*}
then problem 
\begin{equation*}
\left\{\begin{aligned}
(- \Delta)^s\, u & = bu^+ - au^- + |u|^{2_{s}^\ast - 2}\, u && \text{in } \Omega,\\
u & = 0 && \text{in } \R^N \setminus \Omega,
\end{aligned}\right.
\end{equation*}
admits a nontrivial solution.
\end{corollary}
\begin{proof}
Here, one takes~$\mu:= \delta_s$, $\overline s:=s$ and~$s_\sharp:=s$,
and the desired result is a consequence of Theorem~\ref{main1}.
\end{proof}

Now we show how to relate our new results to~\cite[Theorem 1.4]{arXiv.2209.07502}:

\begin{corollary}\label{mix_corollary}
Let~$s\in [{{ 0 }}, 1)$. 
Denote by~$\lambda_l$ the sequence of Dirichlet eigenvalues of the mixed operator~$-\Delta +(-\Delta)^s$, by~$S$ the classical Sobolev constant and
by~$2^* = 2N/(N-2)$ the classical critical Sobolev exponent. 

Then, if~$(a, b)\in Q_l$, $b <\nu_{l-1}(a)$ and
\begin{equation*}
\min \set{a,b} > \lambda_l -\frac{S}{|\Omega|^{2/N}},
\end{equation*}
then problem \begin{equation}\label{mixedab}
\left\{\begin{aligned}
-\Delta \, u + (-\Delta)^s \, u & = bu^+ - au^- + |u|^{2^* - 2}\, u && \text{in } \Omega,\\
u & = 0 && \text{in } \R^N \setminus \Omega,
\end{aligned}\right.
\end{equation}
admits a nontrivial solution.
\end{corollary}

\begin{proof}
We set~$
\mu:= \delta_1 +\delta_s$, where~$\delta_1$ and~$\delta_s$ denote the Dirac measures centered at the points~$1$ and~$s$ respectively. As done in the proof of Corollary~\ref{ILS},
we can take~$\overline s:=1$ and~$s_\sharp:=1$ and deduce the desired result from
Theorem~\ref{main1}.
\end{proof}

Interestingly, our setting is general enough to include also operators containing small terms with the ``wrong'' sign. As a paradigmatic example, we showcase the following result:

\begin{corollary}\label{mix_signed}
Let~$s\in [{{ 0 }}, 1)$ and~$\alpha\in\R$.
Denote by~$\lambda_l$ the sequence of Dirichlet eigenvalues of the mixed operator~$-\Delta - \alpha(-\Delta)^s$, by~$2^* = 2N/(N-2)$ the classical critical Sobolev exponent, and by~$S$ the classical Sobolev constant.

Let~$(a, b)\in Q_l$ and $b <\nu_{l-1}(a)$ and suppose that
\begin{equation*}
\min \set{a,b} > \lambda_l -\frac{S}{|\Omega|^{2/N}},
\end{equation*}

Then, there exists~$\alpha_0>0$, depending only on~$N$, $\Omega$, $a$ and~$b$, such that if~$\alpha\le\alpha_0$, then problem 
\begin{equation*}
\left\{\begin{aligned}
-\Delta \, u - \alpha(-\Delta)^s \, u & = bu^+ - au^- + |u|^{2^* - 2}\, u && \text{in } \Omega,\\
u & = 0 && \text{in } \R^N \setminus \Omega,
\end{aligned}\right.
\end{equation*} admits a nontrivial solution.\end{corollary}

\begin{proof}
We define~$\mu:= \delta_1 -\alpha\delta_s$, $\overline s:=1$ and~$s_\sharp:=1$.
Once again, the desired result follows from
Theorem~\ref{main1}.  
\end{proof}

One more interesting application arises taking $\mu$ as a convergent series of Dirac measures. On this matter, we provide the next two results:

\begin{corollary}\label{serie1}
Let~$1\ge s_0 > s_1> s_2 >\dots \ge 0$. 
Denote by~$\lambda_l$ the sequence of Dirichlet eigenvalues of the operator
\[
\sum_{k=0}^{+\infty} c_k (-\Delta)^{s_k} 
\qquad{\mbox{with $\,c_k\ge 0\,$ and }} \,\sum_{k=0}^{+\infty} c_k \in (0, +\infty), 
\]
by~$S_0$ the best Sobolev constant corresponding to the exponent $s_0$ and
by~$2_{s_0}^* = 2N/(N-2s_0)$ the critical Sobolev exponent. 

Then, if~$(a, b)\in Q_l$, $b <\nu_{l-1}(a)$ and
\begin{equation*}
\min \set{a,b} > \lambda_l -\frac{S_0}{|\Omega|^{2s_0/N}},
\end{equation*}
then problem \begin{equation*}
\left\{\begin{aligned}
\sum_{k=0}^{+\infty} c_k (-\Delta)^{s_k} u & = bu^+ - au^- + |u|^{2_{s_0}^* - 2}\, u && \text{in } \Omega,\\
u & = 0 && \text{in } \R^N \setminus \Omega,
\end{aligned}\right.
\end{equation*}
admits a nontrivial solution.\end{corollary}

\begin{proof}
We set
\[
\mu:=\sum_{k=0}^{+\infty} c_k \,\delta_{s_k} , 
\]
where~$\delta_{s_k}$ denote the Dirac measures centered at each~$s_k$. In this case,
we can take~$\overline s:=0$ and~$s_\sharp:=s_0$ and deduce the desired result from
Theorem~\ref{main1}.
\end{proof}

\begin{corollary}\label{serie2}
Let~$1\ge s_0 > s_1> s_2 >\dots \ge 0$ and~$c_k\in\R$ for all~$k\in\N$ be such that
$$ \sum_{k=0}^{+\infty} c_k \in (0, +\infty).$$
Assume that there exists $\gamma\ge0$ and $\overline k\in\N$ such that 
\begin{equation}\label{serft87604}
c_k>0\ \text{ for all } k\in\{1,\dots, \overline k\}\quad \text{ and }\quad \sum_{k=\overline k +1}^{+\infty} c_k \le \gamma \sum_{k=0}^{\overline k} c_k.
\end{equation}
Denote by~$\lambda_l$ the sequence of Dirichlet eigenvalues of the operator
\[
\sum_{k=0}^{+\infty} c_k (-\Delta)^{s_k} 
\]
by~$S_0$ the best Sobolev constant corresponding to the exponent $s_0$ and
by~$2_{s_0}^* = 2N/(N-2s_0)$ the critical Sobolev exponent. 

Let~$(a, b)\in Q_l$ and~$b <\nu_{l-1}(a)$ and suppose that
\begin{equation*}
\min \set{a,b} > \lambda_l -\frac{S_0}{|\Omega|^{2s_0/N}},
\end{equation*}

Then, there exists~$\gamma_0>0$, depending only on~$N$, $\Omega$, $s_0$, $a$ and~$b$, such that if~$\gamma\in[0,\gamma_0]$
then problem \begin{equation*}
\left\{\begin{aligned}
\sum_{k=0}^{+\infty} c_k (-\Delta)^{s_k} u & = bu^+ - au^- + |u|^{2_{s_0}^* - 2}\, u && \text{in } \Omega,\\
u & = 0 && \text{in } \R^N \setminus \Omega,
\end{aligned}\right.
\end{equation*}
admits a nontrivial solution.\end{corollary}

\begin{proof}
We set
\[
\mu:=\sum_{k=0}^{+\infty} c_k \,\delta_{s_k}
\]
where~$\delta_{s_k}$ denote the Dirac measures centered at each~$s_k$. 

Notice that~\eqref{serft87604} guarantees that
the assumptions on~$\mu$ in~\eqref{mu00},
\eqref{mu3} and~\eqref{mu2} are satisfied.
Thus, we can take~$\overline s:=s_{\overline k}$ and~$s_\sharp:=s_0$ and infer the desired result from
Theorem~\ref{main1}.
\end{proof}

It is worth nothing that the wide generality of our setting enables us to address also the case of the continuous superposition of fractional operators. To be more precise, the following result holds true.

\begin{corollary}\label{function}
Let~$s_\sharp\in [{{ 0 }}, 1)$, $\gamma\ge0$ and~$f$ be a measurable and non identically zero function such that
\begin{equation}\label{fun-int}\begin{split} &
{\mbox{$f\ge 0$ in~$(s_\sharp,1)$,}}\\&
\int_{s_\sharp}^1 f(s) \,ds >0
\\ {\mbox{and }}\qquad&
\int_0^{s_\sharp} \max\{0,-f(s)\} \,ds \le\gamma \int_{s_\sharp}^1 f(s) \,ds .
\end{split}\end{equation}

Denote by~$\lambda_l$ the sequence of Dirichlet eigenvalues of the operator 
\begin{equation}\label{djweiohtgierogh597865432}
\int_0^1 f(s) (-\Delta)^s \, u \, ds ,\end{equation}
 by~$S_\sharp$ the best Sobolev constant corresponding to the exponent $s_\sharp$ and
by~$2_{s_\sharp}^* = 2N/(N-2s_\sharp)$ the fractional critical Sobolev exponent. 

Let~$(a, b)\in Q_l$ and $b <\nu_{l-1}(a)$ and suppose that
\begin{equation*}
\min \set{a,b} > \lambda_l -\frac{S_\sharp}{|\Omega|^{2s_\sharp/N}},
\end{equation*}

Then, there exists~$\gamma_0>0$, depending only on~$N$, $\Omega$, $s_\sharp$, $a$ and~$b$, such that if~$\gamma\in[0,\gamma_0]$
then problem \begin{equation*}
\left\{\begin{aligned}
\int_0^1 f(s) (-\Delta)^s \, u \, ds & = bu^+ - au^- + |u|^{2_{s_\sharp}^* - 2}\, u && \text{in } \Omega,\\
u & = 0 && \text{in } \R^N \setminus \Omega,
\end{aligned}\right.
\end{equation*}
admits a nontrivial solution.\end{corollary}

\begin{proof}
We observe that the operator in~\eqref{djweiohtgierogh597865432}
is a particular case of~$A_\mu$ as defined in~\eqref{AMMU}, where~$d\mu(s)$
boils down to~$f(s)\,ds$.

Additionally, \eqref{fun-int} guarantees that the assumptions
in~\eqref{mu00}, \eqref{mu3} and~\eqref{mu2} are satisfied.
Thus, we can take~$\overline s:=s_\sharp$, which plays the role of the fractional critical exponent. The desired result then follows from Theorem~\ref{main1}.
\end{proof}

\section*{Acknowledgements} 

SD and EV are members of the Australian Mathematical Society (AustMS).
EV is supported by the Australian Laureate Fellowship FL190100081 ``Minimal surfaces, free boundaries and partial differential equations''.

CS is member of INdAM-GNAMPA.

This work was partially completed while KP was visiting the Department of Mathematics and Statistics at the University of Western Australia, and he is grateful for the hospitality of the host department. His visit to the UWA was supported by the Simons Foundation Award 962241 ``Local and nonlocal variational problems with lack of compactness''.

\vfill
\end{document}